\documentclass[12pt]{amsart}
\usepackage{amssymb}
\usepackage[mathscr]{eucal}

\usepackage{colordvi}
\usepackage{graphicx}

\usepackage{rotating}
\usepackage{multirow}

\usepackage{icomma}
\usepackage[ps2pdf]{hyperref}
\usepackage{url}

\usepackage{vmargin}
\setpapersize{USletter}
\setmargrb{.9in}{.9in}{.9in}{.9in} 
 
\numberwithin{equation}{section}
\newtheorem{theorem}{Theorem}[section]
\newtheorem{defn}[theorem]{Definition}
\newtheorem{proposition}[theorem]{Proposition}
\newtheorem{lemma}[theorem]{Lemma}
\newtheorem{cor}[theorem]{Corollary}
\newtheorem{conj}[theorem]{Conjecture}
\newtheorem{sconj}[theorem]{Secant Conjecture}
\newtheorem{example}[theorem]{Example}
\newtheorem{remark}[theorem]{Remark}

\theoremstyle{remark}

\newcounter{FNC}[page]
\def\fauxfootnote#1{{\addtocounter{FNC}{2}\Magenta{$^\fnsymbol{FNC}$}%
     \let\thefootnote\relax\footnotetext{\Magenta{$^\fnsymbol{FNC}$#1}}}}

\newcommand{\Fdot}{F_{\bullet}}

\newcommand{\C}{{\mathbb{C}}}

\renewcommand{\P}{{\mathbb{P}}}

\newcommand{\R}{{\mathbb{R}}}

\newcommand{\W}{\mbox{\rm W}}

\newcommand{\be}{{\bf e}}

\DeclareMathOperator{\rank}{rank}
\DeclareMathOperator{\Span}{span}

\DeclareRobustCommand{\I}{\includegraphics{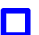}}
\DeclareRobustCommand{\II}{\includegraphics{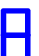}}
\DeclareRobustCommand{\Tw}{\includegraphics{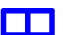}}
\DeclareRobustCommand{\sI}{\includegraphics{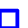}}

\DeclareRobustCommand{\ThI}{\includegraphics{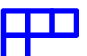}}
\DeclareRobustCommand{\sThI}{\includegraphics{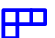}}
\DeclareRobustCommand{\sTT}{\includegraphics{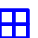}}
\newcommand{\Th}{\includegraphics{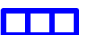}}
\newcommand{\TI}{\includegraphics{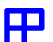}}
\newcommand{\TII}{\includegraphics{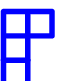}}
\newcommand{\bTI}{\includegraphics{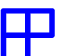}}
\newcommand{\DeCo}[1]{\Blue{#1}}
\newcommand{\demph}[1]{\DeCo{{\sl #1}}}

\newcommand{\M}[1]{\Magenta{#1}}

\title{The Secant Conjecture in the real Schubert calculus}

\author[Garc{\'\i}a-Puente]{Luis D. Garc\'ia-Puente}
\address{Luis Garc\'ia-Puente \\
         Department of Mathematics and Statistics\\
         Sam Houston State University\\
         Huntsville\\
         TX \ 77341\\
         USA}
\email{lgarcia@shsu.edu}
\urladdr{\url{http://www.shsu.edu/~ldg005}}
\author[Hein]{Nickolas Hein}
\address{Nickolas Hein \\
         Department of Mathematics\\
         Texas A\&M University\\
         College Station\\
         Texas \ 77843\\
         USA}
\email{nhein@math.tamu.edu}
\urladdr{\url{http://www.math.tamu.edu/~nhein}}
\author[Hillar]{Christopher Hillar}
\address{Christopher Hillar \\
         Mathematical Sciences Research Institute\\
         17 Gauss Way\\
         Berkeley, CA 94720-5070\\
         USA}
\email{chillar@msri.org}
\urladdr{\url{http://www.msri.org/people/members/chillar}}
\author[Mart\'in del Campo]{Abraham Mart\'in del Campo}
\address{Abraham Mart\'in del Campo \\
         Department of Mathematics\\
         Texas A\&M University\\
         College Station\\
         Texas \ 77843\\
         USA}
\email{asanchez@math.tamu.edu}
\urladdr{\url{http://www.math.tamu.edu/~asanchez}}
\author[Ruffo]{James Ruffo}
\address{James Ruffo \\
         Department of Mathematics, Computer Science, \& Statistics\\
         State University of New York\\
         College at Oneonta\\
         Oneonta, NY 13820\\
         USA}
\email{ruffojv@oneonta.edu}
\urladdr{\url{http://employees.oneonta.edu/ruffojv}}
\author[Sottile]{Frank Sottile}
\address{Frank Sottile \\
         Department of Mathematics\\
         Texas A\&M University\\
         College Station\\
         Texas \ 77843\\
         USA}
\email{sottile@math.tamu.edu}
\urladdr{\url{http://www.math.tamu.edu/~sottile}}
\author[Teitler]{Zach Teitler}
\address{Zach Teitler \\
         Department of Mathematics\\
         Boise State University\\
         Boise\\
         Idaho \ 83725\\
         USA}
\email{zteitler@boisestate.edu}
\urladdr{\url{http://math.boisestate.edu/~zteitler}}

\thanks{Research of Sottile supported in part by NSF grant DMS-070105 and DMS-1001615}
\thanks{Research of Hillar  supported in part by an NSF Postdoctoral Fellowship 
        and an NSA Young Investigator grant}
\thanks{This research conducted in part on computers provided by NSF SCREMS grant DMS-0922866}

\begin{document}

\begin{abstract}
 We formulate the \DeCo{Secant Conjecture}, which is a generalization of the Shapiro
 Conjecture for Grassmannians.
 It asserts that an intersection of Schubert varieties in a Grassmannian is transverse
 with all points real if the flags defining the Schubert varieties are secant along
 disjoint intervals of a rational normal curve.
 We present theoretical evidence for this conjecture as well as computational evidence obtained
 in over one terahertz-year of computing, and we discuss some
 of the phenomena we observed in our data. 
\end{abstract}

\maketitle

%
\section{Introduction}
Some solutions to a system of real polynomial equations are real and the rest
occur in complex conjugate pairs.
While the total number of solutions is determined by the structure of the equations,
the number of real solutions depends rather subtly on the coefficients.
Sometimes there is finer information available in terms of upper bounds~\cite{Kh91,BBS} or
lower bounds~\cite{EG01,SS} on the number of real solutions.
The Shapiro and Secant Conjectures assert the extreme situation of having only real solutions.

The Shapiro Conjecture for Grassmannians posits that if the Wronskian of a vector space of
univariate  
{\sl complex} polynomials has only real roots, then that space is spanned by 
{\sl real} polynomials.
This striking instance of unexpected reality was proven by Eremenko and Gabrielov for
two-dimensional spaces of polynomials~\cite{EG_02,EG10}, 
and the general case was established by Mukhin, Tarasov, and
Varchenko~\cite{MTV_Annals,MTV_JAMS}.
While the statement concerns spaces of polynomials, or more generally the Schubert calculus on
Grassmannians, its proofs  complex
analysis~\cite{EG_02,EG10} and mathematical physics~\cite{MTV_Annals,MTV_JAMS}.
This story was described in the AMS Bulletin~\cite{So_FRSC}.

The Shapiro conjecture first gained attention through partial results and
computations~\cite{So_Shap,Ve}, and further work~\cite{So_fulton} led to 
an extension that appears to hold for flag manifolds, the Monotone Conjecture.
This extension was made in~\cite{RSSS}, which also reported on partial results and 
experimental evidence.
The Monotone Conjecture for a certain family of two-step flag manifolds was proved by
Eremenko, Gabrielov, Shapiro, and Vainshtein~\cite{EGSV}.

The result of~\cite{EGSV} was in fact a proof of reality in the Grassmannian of
codimension-two planes for intersections of Schubert varieties defined with respect to
certain {\sl disjoint secant flags}.
The \demph{Secant Conjecture} postulates an
extension of this result to all Grassmannians.
We give the simplest open instance of the Secant Conjecture.
Let $x_1,\dotsc,x_6$ be indeterminates and consider the polynomial
 \begin{equation}\label{Eq:G25}
   f(s,t,u;x)\ :=\ \det\left(\begin{matrix}
                     1&0&x_1&x_2&x_3\\ 
                     0&1&x_4&x_5&x_6\vspace{3pt}\\
                     1&s&s^2&s^3&s^4\\
                     1&t&t^2&t^3&t^4\\
                     1&u&u^2&u^3&u^4
                   \end{matrix}\right) ,
 \end{equation}
which depends upon parameters $s$, $t$, and $u$.

\begin{conj}\label{C:first}
  Let $s_1<t_1<u_1\,<\, s_2< t_2< \dotsb< u_5\,<\, s_6<t_6<u_6$ be real numbers.
  Then the system of polynomial equations
 \begin{equation}\label{Eq:polysys}
   f(s_i,t_i,u_i; x)\ =\ 0\qquad i=1,\dotsc,6
 \end{equation}
  has five distinct solutions, and all of them are real.
\end{conj}

Geometrically, the equation $f(s,t,u;x)=0$ says that the 2-plane (spanned by the first two
rows of the matrix in~\eqref{Eq:G25}) meets the 3-plane which is secant to the
rational curve $\gamma\colon y\mapsto (1,y,y^2,y^3,y^4)$ at the points
$\gamma(s),\gamma(t),\gamma(u)$.
The hypotheses imply that each of the six 3-planes is secant to $\gamma$ along
an interval $[s_i,u_i]$, and these six intervals are pairwise disjoint. 
The conjecture asserts that all of the 2-planes meeting six 3-planes are real when the
3-planes are secant to the rational normal curve along disjoint intervals.
This statement was true in each of the 285,502 instances we tested.

The purpose of this paper is to explain the Secant Conjecture and its relation to the other
reality conjectures, to describe the data supporting it from a large computational experiment,
and to highlight some other features in our data beyond the Secant Conjecture.
These data may be viewed online~\cite{WWW_secant_Exp}.
We will assume some background on the Shapiro Conjecture as described in the
survey~\cite{So_FRSC} and paper~\cite{RSSS}, and we will not describe the execution
of the experiment, as the methods paper~\cite{Exp-FRSC} presented the software 
framework we have developed for such distributed computational experiments.\smallskip

This paper is organized as follows.
In Section~\ref{Sec:Shapiro}, we present the full 
Secant Conjecture, giving a history of its formulation.
Section~\ref{Sec:SpecialCases} presents some theoretical justification for the Secant
Conjecture as well as a generalization based on limiting cases.
In Section~\ref{S:chords} we analyze the problem of lines meeting all possible
configurations of four secant lines, giving conditions on the secant lines that imply
that both solutions are real.
Section~\ref{S:overlap} describes a statistic, the \demph{overlap number},
which measures the extent of overlap among intervals of secancy.
In Section~\ref{Sec:Exp} we explain the data from our experiment.
About $3/4$ of our over $2$ billion computations did not directly test the Secant
Conjecture, but rather tested geometric configurations that were close to those of the
conjecture. 
Consequently, our data contain much more information than that in support of the Secant
	    Conjecture, and we explore that information in the remaining sections.
Section~\ref{Sec:inner} discusses the lower bounds on the numbers of real solutions we
typically observed for small overlap number, producing a striking \demph{inner border}
in the tabulation of our data.
Finally, in Section~\ref{Sec:gaps}, we discuss Schubert problems with provable lower
bounds and gaps in their numbers of real solutions, a phenomenon we first noticed while
trying to understand our data.

We thank Brian Osserman and the referee for their comments on earlier versions of
this paper.

%
\section{Schubert Calculus and the Secant Conjecture}\label{Sec:Shapiro}

We give background from the Schubert Calculus  necessary to state the Secant Conjecture,
and then we state the equivalent dual Cosecant Conjecture.

%
\subsection{Schubert Calculus}

The Schubert Calculus~\cite{Fu97,FuPr} involves problems of determining the
linear spaces that have specified positions with respect to
other, fixed (flags of) linear spaces.
For example, what are the 3-planes in $\C^7$ meeting 12 given 4-planes non-trivially? 
(There are 462~\cite{Sch1886c}.)
The specified positions are a \demph{Schubert problem}, which determines the number of
solutions. 
The actual solutions depend upon the linear spaces imposing the conditions, or 
\demph{instance} of the Schubert problem.

The \demph{Grassmannian $G(k,n)$} is the set of all $k$-dimensional linear subspaces of
$\C^n$, which is an algebraic manifold of dimension $k(n{-}k)$.
A \demph{flag $\Fdot$} is a sequence of linear subspaces
\[
   \Fdot\ \colon\ F_1\subset F_2\subset\dotsb\subset F_n\,,
\]
where $\dim F_i=i$.
A \demph{partition} $\lambda\colon (n{-}k)\geq \lambda_1\geq\dotsb\geq \lambda_k\geq 0$ is
a weakly decreasing sequence of integers.
A fixed flag $\Fdot$ and a partition $\lambda$ define a 
\demph{Schubert variety $X_\lambda\Fdot$},
\[
   X_\lambda\Fdot\ :=\ 
    \{H\in G(k,n)\mid \dim H\cap F_{n-k+i-\lambda_i}\geq i\quad\mbox{for}\  i=1,\dotsc,k\}\,,
\]
which is a subvariety of codimension 
$|\lambda|:=\lambda_1+\dotsb+\lambda_k$. 
Not every element of the flag is needed to define the Schubert variety.

A \demph{Schubert problem} is a list $\lambda^1,\dotsc,\lambda^m$ of partitions
with $|\lambda^1|+\dotsb+|\lambda^m|=k(n-k)$.
For sufficiently general flags $\Fdot^1,\dotsc,\Fdot^m$, the intersection
\[
   X_{\lambda^1}\Fdot^1\cap X_{\lambda^2}\Fdot^2\cap
   \dotsb\cap X_{\lambda^m}\Fdot^m
\]
is transverse~\cite{Kleiman} and consists of a certain number, 
$d(\lambda^1,\dotsc,\lambda^m)$, of 
points, which may be computed using algorithms in the Schubert Calculus
(see~\cite{Fu97,KL72}).  
(\demph{Transverse} means that at each point of the intersection, the annihilators of the tangent
spaces to the Schubert varieties are in direct sum.)
We write a Schubert problem multiplicatively, 
$\lambda^1\dotsb\lambda^m=d(\lambda^1,\dotsc,\lambda^m)$.
For example, writing $\I$ for the partition $(1,0)$ with $|\I|=1$, we have 
$\I\cdot\I\cdot\I\cdot\I\cdot\I\cdot\I=\I^6=5$ for the Schubert problem on $G(2,5)$ involving
six partitions, each equal to $\I$.
In this notation, Schubert's problem that we mentioned above is $\I^{12}=462$ on $G(3,7)$.

A rational normal curve $\gamma\colon\R\to\R^n$ is affinely equivalent to the
\demph{moment curve}
 \[ 
    \gamma\ \colon\ t\ \longmapsto\ (1, t,\, t^2,\, \dotsc,\, t^{n-1})\,.
 \] 
The \demph{osculating flag $\Fdot(t)$} has $i$-dimensional subspace the 
span of the first $i$ derivatives $\gamma(t),\gamma'(t),\dotsc,\gamma^{(i-1)}(t)$ of
$\gamma$ at $t$.
We state the Theorem of Mukhin, et al.~\cite{MTV_Annals,MTV_JAMS}.

\begin{theorem}[The Shapiro Conjecture]\label{Th:MTV}
  For any Schubert problem $\lambda^1,\dotsc,\lambda^m$ on a Grassmannian $G(k,n)$ and any
  distinct real numbers $t_1,\dotsc,t_m$, the intersection
 \[
   X_{\lambda^1}\Fdot(t_1)\cap X_{\lambda^2}\Fdot(t_2)\cap
   \dotsb\cap X_{\lambda^m}\Fdot(t_m)
 \]
  is transverse and consists of $d(\lambda^1,\dotsc,\lambda^m)$ real points.
\end{theorem}

Transversality is unexpected as osculating flags are not general.

The Shapiro Conjecture concerns intersections of Schubert varieties given
by flags osculating a rational normal curve, and in this form it makes sense for every
flag manifold $G/P$.
Purbhoo showed that it holds for  the orthogonal Grassmannians~\cite{Purbhoo_OG},
but counterexamples are known for other flag manifolds.
There is an appealing version of it---the Monotone Conjecture---that appears to hold for
the classical flag variety~\cite{RSSS}.

%
\subsection{The Secant Conjecture}

Eremenko, et al.~\cite{EGSV} proved a generalization of the Monotone Conjecture for
flags consisting of a codimension-two plane lying on a hyperplane, where 
it becomes a statement about real rational functions.
Their theorem asserts that a Schubert problem on $G(n{-}2,n)$ has only real
solutions if the flags satisfy a special property that we now describe.
A flag $\Fdot$ of linear subspaces is 
\demph{secant along an interval $I$} of a
rational normal curve $\gamma$ if every subspace in the flag is spanned by its
intersection with $I$.
This means that there are distinct points 
$t_1,\dotsc,t_{n-1}\in I$ such that for each $i=1,\dotsc,n{-}1$, the subspace $F_i$ of the
flag $\Fdot$ is spanned by $\gamma(t_1),\dotsc,\gamma(t_i)$.

\begin{sconj}\label{SecConj}
  For any Schubert problem $\lambda^1,\dotsc,\lambda^m$ on a Grassmannian $G(k,n)$ and any
  flags $\Fdot^1,\dotsc,\Fdot^m$ that are secant to a rational normal curve $\gamma$ along
  disjoint intervals, the intersection
\[
   X_{\lambda^1}\Fdot^1\cap X_{\lambda^2}\Fdot^2\cap
   \dotsb\cap X_{\lambda^m}\Fdot^m
\]
 is transverse and consists of $d(\lambda^1,\dotsc,\lambda^m)$ real points.
\end{sconj}

Conjecture~\ref{C:first} is the case of this Secant Conjecture for
the Schubert problem $\I^6=5$ on $G(2,5)$.
The Schubert variety $X_{\sI}\Fdot$ is 
\[
   X_{\sI}\Fdot\ =\ \{H\in G(2,5)\mid   \dim H\cap F_3\geq 1\}\,;
\]
that is, the set of 2-planes meeting a fixed 3-plane non-trivially.
Since $F_4$ and $F_5$ are irrelevant we drop them from the flag and refer to $F_3$ and
$X_{\sI}F_3$. 
For every Schubert condition, there is a largest element of the flag imposing a relevant
condition; call this the \demph{relevant subspace}. 
The relevant subspace in this example is $F_3$.

For $s,t,u\in\R$, let $F_3(s,t,u)$ be the linear span of 
$\gamma(s)$, $\gamma(t)$, and $\gamma(u)$, a 3-plane secant to $\gamma$ 
with points $\gamma(s)$, $\gamma(t)$, and $\gamma(u)$ of secancy.
Thus, the condition $f(s,t,u;x)=0$ of Conjecture~\ref{C:first} implies that the linear span $H$ 
of the first two rows of the matrix in~\eqref{Eq:G25}---a general 2-plane in
5-space---meets the linear span $F_3(s,t,u)$ of the last three rows.
Thus
\[
    f(s,t,u;x)\ =\ 0
   \quad\Longleftrightarrow\quad
    H\in X_{\sI}F_3(s,t,u)\,.
\]
Lastly, the condition on the ordering of the points $s_i,t_i,u_i$ in
Conjecture~\ref{C:first} implies that the six 
flags $F_3(s_i,t_i,u_i)$ are secant along disjoint intervals.

%
%
\subsection{Grassmann Duality and the Cosecant Conjecture}
\label{Sec:Duality}

Associating a linear subspace $H$ of a vector space $V\simeq \C^n$ to its annihilator
$\delta(H):=H^\perp\subset V^*$ induces an isomorphism $\delta\colon G(k,n)\to G(n{-}k,n)$
called \demph{Grassmann duality}.
This notion extends to flags and the dual of an osculating flag is an osculating flag.
Secancy is not preserved under duality.
We next formulate the (equivalent) dual statement to the Secant Conjecture, which we call the 
Cosecant Conjecture.

Grassmann duality respects Schubert varieties.
Given a flag $\Fdot\subset\C^n$, let $\Fdot^\perp$ be the flag whose $i$-dimensional
subspace is $F^\perp_i:=(F_{n-i})^\perp$.
Then 
\[
   \delta(X_\lambda\Fdot)\ =\ X_{\lambda^T}\Fdot^\perp\,,
\]
where $\lambda^T$ is the \demph{conjugate partition} to $\lambda$.
For example,
\[
   \Tw^T\ =\ \raisebox{-6pt}{\II}\,,
   \qquad
   \raisebox{-6pt}{$\bTI^T$}\ =\ \raisebox{-6pt}{$\bTI$}\,,
   \qquad\mbox{and}\qquad
   \raisebox{-6pt}{$\ThI^T$}\ =\ \raisebox{-13pt}{$\TII$}\,.
\]
That is, if we represent $\lambda$ by its \demph{Young diagram}---a left-justified array of boxes
with $\lambda_i$ boxes in row $i$---then the diagram of $\lambda^T$ is the matrix-transpose of
the diagram of $\lambda$.

If $\gamma(t)=(1,t,t^2,\dotsc,t^{n-1})$ is the rational normal curve, then the dual of the
family $F_{n-1}(t)$ of its osculating $(n{-}1)$-planes is a curve
$\gamma^\perp(t):=(F_{n-1}(t))^\perp$, which is
\[
   \gamma^\perp(t)\ =\ 
   \bigl(\tbinom{n-1}{n-1}(-t)^{n-1}\,,\, \dotsc\,,\,
          -{\tbinom{n-1}{3}t^3}\,,\, \tbinom{n-1}{2}t^2\,,\,
          -{(n{-}1)t}\,,\,1\,\bigr)\,,
\]
in the basis dual to the standard basis.
Moreover, $(F_{n-k}(t))^\perp$ is the osculating $k$-plane to this dual rational normal curve
$\gamma^\perp$ at the point $\gamma^\perp(t)$.
Thus Grassmann duality preserves Schubert varieties given by flags osculating the rational
normal curve, and 
the dual statement to Theorem~\ref{Th:MTV} is simply itself.

This is however not the case for secant flags.
The general secant $(n{-}1)$-plane
\[
   F_{n-1}(s_1,s_2,\dotsc,s_{n-1}) =\ 
   \Span\{ \gamma(s_1)\,,\,\gamma(s_2)\,,\,\dotsc\,,\,\gamma(s_{n-1})\}\,,
\]
secant to $\gamma$ at the points $\gamma(s_1),\dotsc,\gamma(s_{n-1})$, has dual space 
spanned by the vector 
\[
   \bigl( (-1)^{n-1}e_{n-1}\,,\, \dotsc\,,\,
    -{e_3}\,,\, e_2\,,\, -{e_1}\,,\, 1\bigr)\,,
\]
where $e_i$ is the $i$th elementary symmetric function in the parameters $s_1,\dotsc,s_{n-1}$.
This dual space is not secant to the dual rational normal curve $\gamma^\perp$.

In general, a \demph{cosecant subspace} is a subspace that is dual to a secant subspace.
If 
\[
   F_k(s_1,s_2,\dotsc,s_k)\ =\ 
    \Span\{\gamma(s_1)\,,\,\gamma(s_2)\,,\,\dotsc\,,\,\gamma(s_k)\}\,, 
\]
then the corresponding cosecant subspace is
\[
   F_{n-1}^\perp(s_1)\,\cap\,
   F_{n-1}^\perp(s_2)\,\cap\, \dotsb\,\cap\,
   F_{n-1}^\perp(s_k)\,,
\]
the intersection of $k$ hyperplanes osculating the rational normal curve $\gamma^\perp$. 
A \demph{cosecant flag} is a flag whose subspaces are cut out by hyperplanes osculating
$\gamma$.
It is \demph{cosecant along an interval} of $\gamma$ if these hyperplanes osculate $\gamma$ at points
of the interval. 

Thus, under Grassmann duality the Secant Conjecture for
$G(n{-}k,n)$ becomes the following equivalent \demph{Cosecant Conjecture} for $G(k,n)$.

\begin{conj}[Cosecant Conjecture]
 For any Schubert problem $\lambda^1,\dotsc,\lambda^m$ on a Grassmannian $G(k,n)$ and any flags
 $\Fdot^1,\dotsc,\Fdot^m$ that are cosecant to a rational normal curve $\gamma$ along disjoint
 intervals, the intersection 
\[
   X_{\lambda^1}\Fdot^1\,\cap\,
   X_{\lambda^2}\Fdot^2\,\cap\,\dotsb\,\cap\, X_{\lambda^m}\Fdot^m
\]
 is transverse and consists of $d(\lambda^1,\dotsc,\lambda^m)$ real points.
\end{conj}

%
\section{Some special cases of the Secant Conjecture}\label{Sec:SpecialCases}

A degree of justification for posing the Secant Conjecture
is provided by the history of its development from the Shapiro and Monotone Conjectures,
as this shows its connection to
proven results and established conjectures, and its validity for $G(n{-}2,n)$~\cite{EGSV}.
Here, we give more concrete justifications, which include proofs in some special cases.

%
\subsection{Arithmetic progressions of secancy}\label{SSSec:MTV}
Fix a parametrization $\gamma\colon\R\to\R^n$ of a rational normal curve.
For $t\in\R$ and $h>0$, let $\Fdot^h(t)$ be the flag whose $i$-dimensional subspace is
\[
   F_i^h(t)\ :=\ 
    \Span\{\gamma(t),\, \gamma(t{+}h),\, \dotsc,\,\gamma(t{+}(i{-}1)h)\}\,,
\]
which is spanned by an arithmetic progression of length $i$ 
with step size $h$.
Work of Mukhin, et al.~\cite{MTV_XXX} implies the Secant Conjecture for the Schubert
problem
 \begin{equation}\label{Eq:Schubert_number}
   \I^{k(n-k)}\ =\ [k(n{-}k)]!\frac{1!2!\dotsb(k{-}1)!}{(n{-}k)!\dotsb(n{-}2)!(n{-}1)!}
 \end{equation}
for such secant flags.

Let $\C_{n-1}[t]$ be the space of polynomials of degree at most $n{-}1$.
The discrete Wronskian with step size $h$ of polynomials $f_1,\dotsc,f_k$
is the determinant
 \begin{equation}\label{Eq:Wr_det}
   \W_h(f_1,f_2,\dotsc,f_k)\ :=\ \det\left(  \begin{matrix}
           f_1(t) & f_1(t+h) & \dotsb & f_1(t+(k{-}1)h)\\
           f_2(t) & f_2(t+h) & \dotsb & f_2(t+(k{-}1)h)\\
           \vdots& \vdots  & \ddots & \vdots \\
           f_k(t) & f_k(t+h) & \dotsb & f_k(t+(k{-}1)h)\\
    \end{matrix}\right)\ .
 \end{equation}
For general $f_1,\dotsc,f_k\in\C_{n-1}[t]$, this polynomial has degree $k(n{-}k)$.
Up to a scalar, the polynomial $W_h$ depends
only on the linear span of the polynomials $f_1,\dotsc,f_k$,
giving a map
\[
   \W_h\ \colon\ G(k, \C_{n-1}[t])\ \longrightarrow\ \P^{k(n-k)}\,,
\]
where $\P^{k(n-k)}$ is the projective space of polynomials of degree
at most $k(n-k)$.
Mukhin, et al.~\cite{MTV_XXX} show that $\W_h$ is a finite
map.
It is a linear projection of the Grassmannian  in its
Pl\"ucker embedding, so the fiber over a general polynomial
$w(t)\in\P^{k(n-k)}$ consists of  $d(\I^{k(n-k)})$ reduced points,
each of which is a space $V$ of polynomials with discrete Wronskian
$w(t)$.
As a special case of Theorem~2.1 in~\cite{MTV_XXX}, we have the following
statement.

\begin{proposition}\label{MTV}
  Let $V\subset\C_{n-1}[t]$ be a $k$-dimensional space of polynomials
  whose discrete Wronskian $\W_h(V)$ has distinct real roots
  $z_1,\dotsc,z_N$, each of multiplicity $1$.
  If for all $i\neq j$, we have $|z_i-z_j|\geq h$, 
 then the space $V$ has a basis of real polynomials.
\end{proposition}

\begin{cor}
  Set $N:=k(n{-}k)$ and suppose that $\Fdot^h(z_1),\dotsc,\Fdot^h(z_N)$ are disjoint secant
  flags with $z_i+(n{-}1)h< z_{i+1}$ for each $i=1,\dotsc,N-1$.
  Then the intersection 
 \begin{equation}\label{Eq:arith}
   X_{\sI}\Fdot^h(z_1)\,\cap\,
   X_{\sI}\Fdot^h(z_2)\,\cap\,\dotsb\,\cap\,
   X_{\sI}\Fdot^h(z_N)
 \end{equation}
  in $G(n{-}k,n)$ is transverse with all points real.
\end{cor}

\begin{proof}
  We identify points in the intersection~\eqref{Eq:arith} with the fibers of the discrete
  Wronski map 
  $W_h$ over the polynomial $(t-z_1)\dotsb(t-z_{k(n-k)})$, which will prove reality.
  Transversality follows by an argument of Eremenko and Gabrielov given
  in~\cite[Ch.~13]{IHP}: a finite analytic map between complex manifolds that
  has only real points in its fibers above an open set of real points is necessarily
  unramified over those points.

  A polynomial of degree $n{-}1$ is the composition of the
  parametrization $\gamma\colon\C\to\C^n$ of the rational normal curve with a linear form 
  $\C^n\to\C$.
  In this way, a subspace $V$ of polynomials of dimension $k$ corresponds to a surjective
  map $V\colon\C^n\to\C^k$.
  We will identify such a map with its kernel $H$, which is a point in $G(n{-}k,n)$.

  The column space of the matrix in~\eqref{Eq:Wr_det} is the image under $V$ of the linearly
  independent vectors $\gamma(t),\gamma(t+h),\dotsc,\gamma(t+(k{-}1)h)$.
  These vectors span $F_{k}^h(t)$.
  Thus the determinant $W_h(V)$ vanishes at a point $t$ exactly when the map
\[
   V\ \colon\ F_{k}^h(t)\ \longrightarrow\ \C^{k}
\]
  does not have full rank; that is, when 
\[
   \dim H \cap\ F_{k}^h(t)\ \geq\ 1\,,
\]
  which is equivalent to $H\in X_{\sI}\Fdot^h(t)\subset G(n{-}k,n)$.

  It follows that  points in the intersection~\eqref{Eq:arith}
  correspond to $k$-dimensional spaces of polynomials $V$ with discrete Wronskian 
  $(t-z_1)\dotsb(t-z_{k(n-k)})$, and each of these are real, by
  Proposition~\ref{MTV}.
\end{proof}

%
\subsection{The Shapiro Conjecture is the limit of the Secant Conjecture}\label{SSSec:Shapiro}

The osculating plane $F_i(s)$ is the unique $i$-dimensional plane having maximal order of
contact with the rational normal curve $\gamma$ at the point $\gamma(s)$.
This implies that it is a limit of secant planes, and in fact every limit of secant planes in
which the points come together is an osculating plane. 

\begin{lemma}
 Let $\{s_1^{(j)},\dotsc,s_i^{(j)}\}$ for $j=1,2,\dotsc$ be a sequence of lists of $i$
 distinct complex numbers that all converge to the same number, 
 $\lim_{j\to\infty} s_p^{(j)}=s$, for each $p=1,\dotsc,i$ and for some number $s$.
 Then
\[
   \lim_{j\to\infty} \Span\{\gamma(s_1^{(j)}),\gamma(s_2^{(j)}),\dotsc,\gamma(s_i^{(j)})\}\ =\ F_i(s)\,.
\]
\end{lemma}

As transversality and reality are preserved under perturbation, we conclude 
that Theorem~\ref{Th:MTV} is a limiting case of the Secant Conjecture.
Conversely, Theorem~\ref{Th:MTV} implies the following.

\begin{theorem}\label{Th:limit}
 Let $\lambda^1,\dotsc,\lambda^m$ be a Schubert problem and $t_1,\dotsc,t_m$ be distinct
 points of the rational normal curve $\gamma$.
 Then there exists an $\epsilon>0$ such that if for each $i=1,\dotsc,m$, 
 $\Fdot^i$ is a flag secant to $\gamma$
 along an interval of length $\epsilon$ containing $t_i$, then the intersection
 \begin{equation}\label{Eq:intL}
  X_{\lambda^1}\Fdot^1\;\cap\; X_{\lambda^2}\Fdot^2\;\cap\ 
     \dotsb\ \cap\;X_{\lambda^m}\Fdot^m
 \end{equation}
 is transverse with all points real.
\end{theorem}

This implies that for generic secant flags $\Fdot^1,\dotsc,\Fdot^m$, the
intersection~\eqref{Eq:intL} is transverse, which implies that secant flags are sufficiently
general for the Schubert Calculus.
Furthermore, Theorem~\ref{Th:limit} reduces the Secant Conjecture~\ref{SecConj} to its
transversality statement. 

%
\subsection{Generalized Secant Conjecture}\label{S:general}

Theorem~\ref{Th:limit} suggests a 
conjecture involving flags that are intermediate
between secant and osculating, and which includes the Secant Conjecture and 
Theorem~\ref{Th:MTV} as special cases.

A \demph{generalized secant subspace} to the rational normal curve
$\gamma$ is spanned by osculating subspaces of $\gamma$.
This notion includes secant subspaces, for a one-dimensional subspace that osculates
$\gamma$ is simply one that is spanned by a point of $\gamma$.
A flag $\Fdot$ is \demph{generalized secant} to $\gamma$ if each of the linear spaces in
$\Fdot$ are generalized secant subspaces.
A generalized secant flag is \demph{secant along an interval} of $\gamma$ if the osculating
subspaces that span its linear spaces osculate $\gamma$ at points of the interval.

\begin{conj}[Generalized Secant Conjecture]\label{Co:GSC}
  For any Schubert problem $\lambda^1,\dotsc,\lambda^m$ on a Grassmannian $G(k,n)$ and any
  generalized secant flags $\Fdot^1,\dotsc,\Fdot^m$ that are secant to a rational normal
  curve $\gamma$ along disjoint intervals, the intersection
\[
   X_{\lambda^1}\Fdot^1\cap X_{\lambda^2}\Fdot^2\cap
   \dotsb\cap X_{\lambda^m}\Fdot^m
\]
 is transverse and consists of $d(\lambda^1,\dotsc,\lambda^m)$ real points.
\end{conj}

This includes the Secant Conjecture as the case when all of the flags
are secant flags, but it also includes Theorem~\ref{Th:MTV}, which 
is when all flags are osculating.
Many of the computations in our 
experiment tested instances of this conjecture where one or two flags were 
osculating while the rest were secant flags.
This choice was made to make the computation feasible for some Schubert problems.

There is also a Generalized Cosecant Conjecture and a corresponding version of
Theorem~\ref{Th:limit}, which we do not formulate. 

%
\section{The problem of four secant lines}\label{S:chords}

We give an in-depth look at 
the Schubert problem $\I^4=2$ on $G(2,4)$ where $\I$ denotes the
Schubert condition that a two-plane in $\C^4$ meets a fixed two-plane nontrivially.
Equivalently,  $\I^4=2$ is the Schubert problem of lines in $\P^3$ that meet four fixed
lines. 
Let $\gamma: \R \rightarrow \P^3$ be a rational normal curve.
We consider the lines in $\P^3$ that meet four lines $\ell_1,\ell_2,\ell_3,\ell_4$ which are 
secant to $\gamma$.

For $s_1, s_2\in\R$ let $\ell(s_1,s_2)$ denote the secant line to $\gamma$ through
$\gamma(s_1)$, $\gamma(s_2)$. 
Given $s_1 < \dotsb < s_8$, the Secant Conjecture (which is in this case a theorem of
Eremenko, et al.~\cite{EGSV}) asserts that both lines
meeting the four fixed lines
\[  
    \ell(s_1,s_2)\,,\ \  \ell(s_3,s_4)\,,\ \ \ell(s_5,s_6)\,,\ \ \ell(s_7,s_8)  
\]
are real.
We investigate phenomena beyond the Secant Conjecture by 
letting $\rho$ be a permutation of $\{1,\dotsc,8\}$ and
taking $\ell_1,\ell_2,\ell_3,\ell_4$ to be 
$\ell(s_{\rho(1)}, s_{\rho(2)}), \dotsc, \ell(s_{\rho(7)}, s_{\rho(8)})$.

There are 17 combinatorial configurations of four secant lines along $\gamma\simeq S^1$. 
These are indicated by the chord diagrams in Table~\ref{T:overlap numbers}, which shows 
the number of real solutions found when we computed $100,000$ instances of each
configuration.
\begin{table}[htb]
\caption{Configurations of four secant lines with results of an experiment.}
\label{T:overlap numbers}
\begin{tabular}{c r|r||r|r||r|r|r|r|r||}
\cline{3-10}
& &
 \includegraphics[height=38pt]{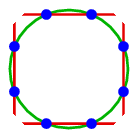}&
 \includegraphics[height=38pt]{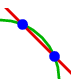}&
 \includegraphics[height=38pt]{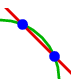}&
 \includegraphics[height=38pt]{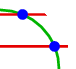}&
 \includegraphics[height=38pt]{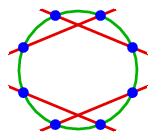}&
 \includegraphics[height=38pt]{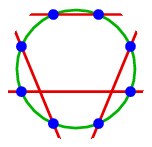}&
 \includegraphics[height=38pt]{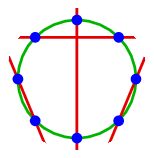}&
 \includegraphics[height=38pt]{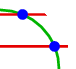}\\ \cline{2-10}
real& \multicolumn{1}{|c||}{0}  &    &   &23723&     &    &     &     &29398 \\ \cline{2-10}
roots& \multicolumn{1}{|c||}{2}  & 100000&100000&76277&100000&100000&100000&100000&70602 \\\cline{2-10}

\end{tabular}\vspace{12pt}

\begin{tabular}{r |r|r|r|r||r|r||r|r||r||}\cline{2-10}
& \includegraphics[height=38pt]{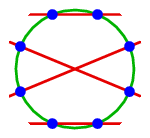}&
 \includegraphics[height=38pt]{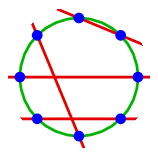}&
 \includegraphics[height=38pt]{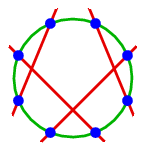}&
 \includegraphics[height=38pt]{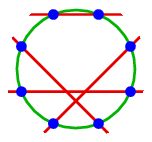}&
 \includegraphics[height=38pt]{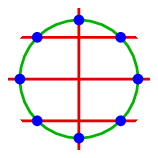}&
 \includegraphics[height=38pt]{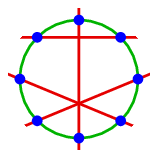}&
 \includegraphics[height=38pt]{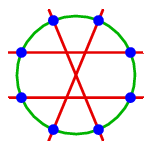}&
 \includegraphics[height=38pt]{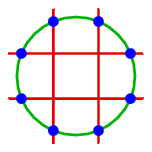}&
 \includegraphics[height=38pt]{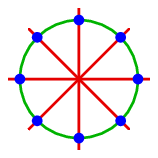}\\\hline
 \multicolumn{1}{|c||}{0} &     &     &     &52395&     &      &     &65783&      \\\hline
 \multicolumn{1}{|c||}{2} & 100000&100000&100000&47605&100000&100000 &100000 &34217& 100000 \\\hline
\end{tabular}


\end{table}
For most configurations, we only observed real solutions, and 
in only four configurations did we find any non-real solutions.
We will give a simple explanation of this observation.

Counting constants shows there is a unique doubly-ruled quadric surface $Q$
that contains the lines $\ell_1$, $\ell_2$, and $\ell_3$ in one ruling, as shown in 
Figure~\ref{F:secant}.
\begin{figure}[htb]
  \begin{picture}(290,125)(0,8)
   \put(0,0){\includegraphics[height=135pt,viewport=5 77 430 282,clip]{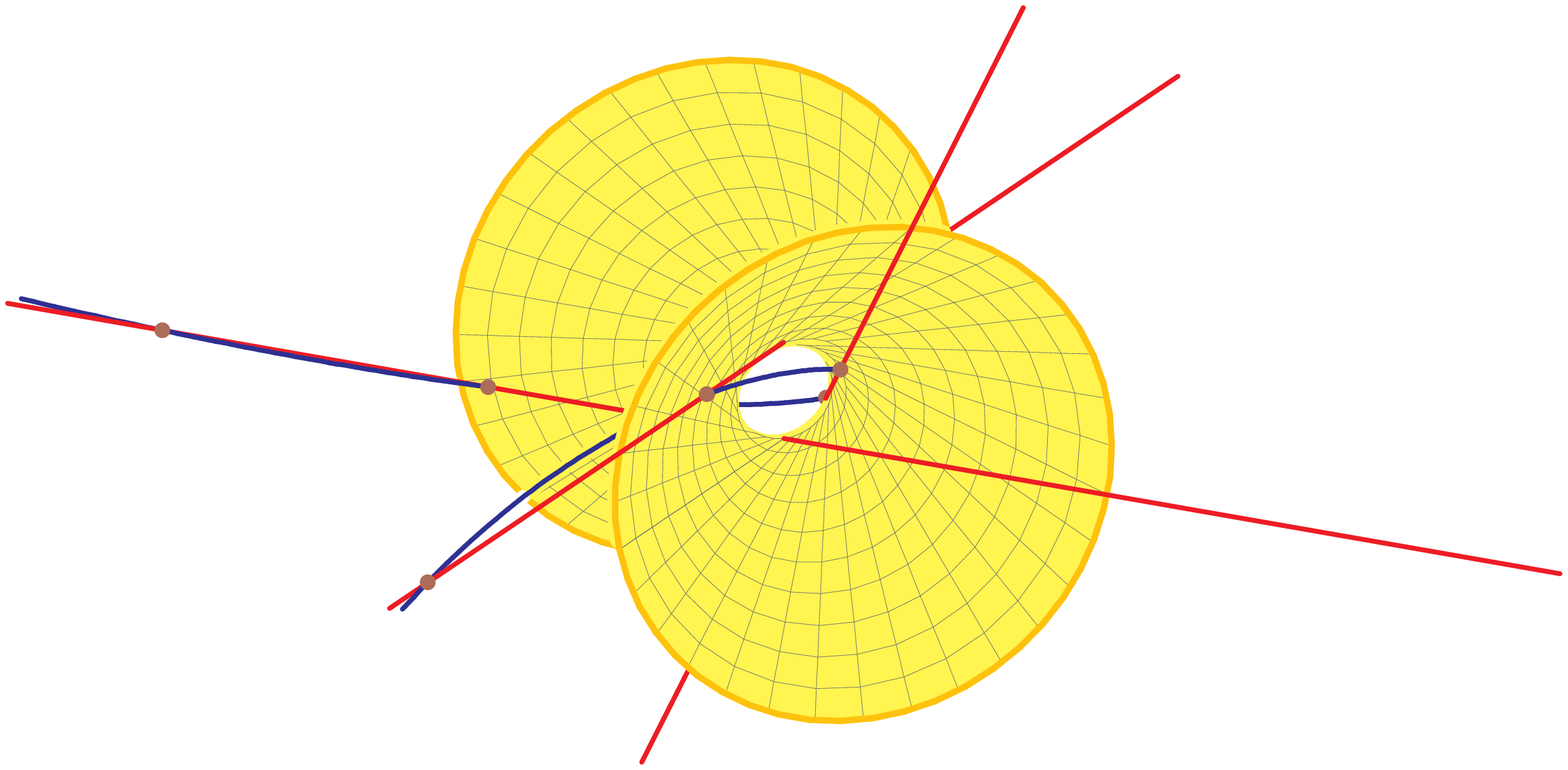}}
   \thicklines
   \put(105,43){$\gamma$}\put(115,44){\vector(2,-1){25}}
   \put(97,0){$\ell_1$}
   \put(281,113){$\ell_2$}
   \put(281,37){$\ell_3$}
  \end{picture}
\caption{Quadric through three secant lines.}
\label{F:secant}
\end{figure}
The two lines of the second ruling of $Q$ through the two points of intersection of $\ell_4$
with $Q$ are the solutions to the Schubert problem $\I^4 = 2$ for these four secant lines.

The quadric $Q$ divides its complement in $\R\P^3$ into two connected components
(the domains where the quadratic form is positive or negative), called the
\demph{sides} of $Q$. 
Three lines $\ell_1$, $\ell_2$, $\ell_3$ give six points of secancy which are the intersections
of $\gamma$ with $Q$ and which divide $\gamma$ into six segments
that alternate between the two sides of $Q$.
If the fourth secant line $\ell_4$ has its two points of secancy lying on opposite sides
of $Q$, then $\ell_4$ has a real intersection with $Q$,
so that the Schubert problem has one (and hence two) real solutions.
The points of secancy of $\ell_4$ lie on opposite sides of $Q$ if in the interval between the
two points of secancy, the curve $\gamma$ crosses $Q$ an odd number of times. 
That is, the interval contains an odd number of points of secancy of the lines $\ell_1$,
$\ell_2$, and $\ell_3$.

This simple topological argument shows that if at least one of the four secant lines has
such an odd interval of secancy, then the Schubert problem will have only real solutions,
independently of the actual positions of the secant lines. 
Twelve of the 17 configurations have at least one odd interval of secancy, and
therefore will always give two real solutions.
Four configurations with only even intervals of secancy were observed to have 
either zero or two real solutions.
Only the configuration with disjoint intervals of secancy has even intervals of secancy,
and yet has only real solutions. 
This deeper fact was proven in~\cite{EGSV}.

%
\section{Overlap number}\label{S:overlap}
For most Schubert problems, the number of different configurations of secant flags 
is astronomical. 
Consider the problem $\I^4\cdot \sThI^2=12$ on the Grassmannian
of $3$-planes in $7$-space.
The condition $\I$ has relevant subspace
$F_4$ and the condition $\sThI$ has relevant subspace $F_5$.
The resulting 26 points of secancy have at least
\[
   \left\lceil\binom{26}{4,4,4,4,5,5}\cdot\frac{1}{4!}\cdot
       \frac{1}{2!}\cdot\frac{1}{26}\cdot\frac{1}{2}\right\rceil\
   =\  
   3,381,948,761,563
\]
combinatorially different configurations.
To cope with this complexity,
 we introduce a statistic on these configurations---the overlap
  number---which  is zero if and only if the flags are disjoint, and we tabulate the
results of our experiment using this statistic.

In an instance of a Schubert problem $\lambda^1,\dotsc,\lambda^m$ with relevant subspaces
of respective dimensions $i_1,\dotsc,i_m$, to define the relevant
subspaces of the $j$th secant flag, 
\[
   F^j_1\ \subsetneq\ F^j_2\ \subsetneq\ \dotsb\  \subsetneq\ F^j_{i_j}\,,
\]
we need a choice of an ordered set $T_j$ of $i_j$ points of $\gamma$.
The overlap number measures how much these sets of points 
$T_1,\dotsc,T_m\subset\gamma$ overlap.

Let $T$ be their union.
Since $\gamma$ is topologically a circle, removing a point $p\in\gamma\setminus T$, 
we may assume that $T_1,\dotsc,T_m\subset\R$.
Each set $T_j$ defines an interval $I_j$ of $\R$ and we let $o_j$ be the number of points of
$T\setminus T_j$ lying in $I_j$. 
This sum $\DeCo{\Sigma}:=o_1+\dotsb+o_m$ depends upon $p\in\gamma\setminus T$, and the 
\demph{overlap number} is the minimum of these sums as $p$ varies.

For example, consider a Schubert problem with relevant subspaces of dimensions $3$, $2$, and
$2$.
Suppose that we have chosen seven points on $\gamma$ in groups of $3$, $2$, and $2$.
This is represented schematically on the left in Figure~\ref{F:overlap}, in which $\gamma$
\begin{figure}[htb]
\caption{Computation of overlap number.}
\label{F:overlap}
\[
   \raisebox{-42pt}{%
     \begin{picture}(90,90)
      \put(0,0){\includegraphics{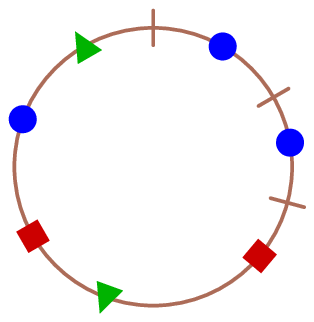}}
      \put(40,72){$p_1$}
      \put(64,60){$p_2$}
      \put(68,35){$p_3$}
    \end{picture}}
   \qquad\qquad
  \begin{tabular}{|c|r|c|c|c|c|}\hline
    &&\includegraphics{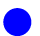}&
     \includegraphics{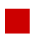}&\includegraphics{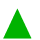}&$\Sigma$\\\hline
    \raisebox{1pt}{$p_1$}&\includegraphics{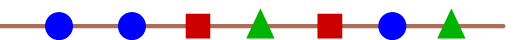}&3&1&2&6\\\hline
    \raisebox{1pt}{$p_2$}&\includegraphics{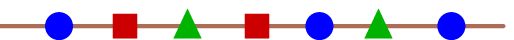}&4&1&2&7\\\hline
    \raisebox{1pt}{$p_3$}&\includegraphics{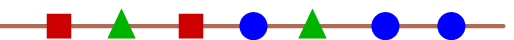}&1&1&2&4\\\hline
  \end{tabular}
\]
\end{figure}
is a circle, and the points in the sets $T_1$, $T_2$, and $T_3$ are represented by 
circles (\includegraphics{figures/cir.eps}), squares (\includegraphics{figures/sqr.eps}),
and triangles (\includegraphics{figures/tri.eps}), respectively.
For each of three points $p_1$, $p_2$, and $p_3$ of $\gamma$, we compute the number $o_i$
and their sum $\Sigma$, displaying the results in the table on the right-hand side of 
Figure~\ref{F:overlap}.
The minimum of the sum $\Sigma$ for all choices of points is achieved by $p_3$.

If one (or more) of the flags are osculating, we compute the overlap number by treating
the point of osculation as a point with multiplicity equal to the dimension of the relevant 
subspace.

%
\section{Experimental evidence for the Secant Conjecture}\label{Sec:Exp}

We tested the Secant Conjecture by conducting a massive experiment whose data are 
available on-line~\cite{WWW_secant_Exp}.
This experiment used symbolic exact arithmetic to compute the number of real solutions
for specific instances of Schubert problems.
These computations are possible because Schubert problems are readily modeled on a
computer, and for those of moderate size, we may algorithmically determine the number of
real solutions with software tools.
Our experiment primarily used the mathematical software Singular~\cite{SINGULAR} and Maple
(see \cite{Exp-FRSC} for further details about the implementation of the computations,
including a comprehensive list of software tools used).
If the software is  reliably implemented, which we believe, then this computation provides a
proof that the given instance has the computed number of real solutions.
This procedure may be semi-automated and run on supercomputers (as described
in~\cite{Exp-FRSC}), which allows us to amass the 
considerable evidence we have collected in support of the Secant Conjecture.

\subsection{Experimental data}
Table~\ref{T:number_problems} shows how many Schubert problems on each Grassmannian of
$k$-planes in $n$-space had been studied when we halted the experiment on 26 May 2010.
\begin{table}[htb]
\caption{Schubert problems studied}
\label{T:number_problems}
 \begin{tabular}{|c||c|c|c|c|c|}\hline
  $k\backslash n{-}k$ & 2&3&4&5&6\\\hline\hline
     2 & 1 & 5 & 22 & 81 & 55 \\\hline
     3 & 5 & 64 & 114 & 79 & \\\hline
     4 &22 & 107 & 67 & &\\\hline
     5 & 81 & & & & \\\hline
 \end{tabular}
\end{table}
Our experiment not only tested the Secant Conjecture
but also studied the relationship between the overlap number
and the number of real solutions 
for many Schubert problems on small Grassmannians.
We  computed $2,058,810,000$ 
instances of 703 Schubert problems.
About one-fourth of these ($498,737,669$) 
were instances of the (Generalized) Secant
Conjecture, and the rest involved non-disjoint secant flags.  
The Generalized Secant Conjecture held in every computed instance.
The remaining $1,560,072,331$ 
instances involved secant flags with 
some overlap in their intervals of secancy,
measured by the overlap number.

The experiment computed Schubert problems using either 0, 1, or 2 osculating flags, with
the rest secant flags.
In the on-line database~\cite{WWW_secant_Exp}, this number of osculating flags determines the 
\demph{computation type} which is 1, 2, or 3 for 0, 1, or 2 osculating flags.
The experiment used randomly chosen flags, which were generated using random generator seeds
that are stored in our database, so that all computations are reproducible.

Table~\ref{T:X^4X31^2=12} shows part of the data we obtained testing the full Secant
Conjecture for the Schubert problem $\I^4\cdot\sThI^2=12$ on $G(3,7)$.
\begin{table}[htb]
\caption{Experimental data for $\I^4\cdot \sThI^2=12$ with all secant flags.}
\label{T:X^4X31^2=12}
\noindent{\small 

\noindent\begin{tabular}{r|r||r|r|r|r|r|r|r|r|r|c||r|}
\multicolumn{13}{c}{Overlap Number}\\
\cline{2-13}
\multirow{11}{*}{\begin{sideways}Real Solutions\end{sideways}} &
\textbackslash & 0&1&2&3&4&5&6&\ $\dotsb$\ &9&\ $\dotsb$\
&Total\\\cline{2-13}\noalign{\smallskip}\cline{2-13}
& 0&       & &       &      &      &      &      &$\dotsb$&     1&$\dotsb$&     691\\\cline{2-13}
& 2&       & &       &      &      &     9&     7&$\dotsb$&     8&$\dotsb$&   72857\\\cline{2-13}
& 4&       & &       &      &    79&   917&  1990&$\dotsb$&   524&$\dotsb$&  523362\\\cline{2-13}
& 6&       & &       &   814&  5713& 12550& 18330&$\dotsb$&  4531&$\dotsb$& 1418911\\\cline{2-13}
& 8&       & &       &   635&  4646& 15947& 17180&$\dotsb$&  6055&$\dotsb$& 1983639\\\cline{2-13}
&10&       & &       &  1226&  6912& 18403& 17236&$\dotsb$&  6801&$\dotsb$& 1649923\\\cline{2-13}
&12&2320873& &  51120& 99413&206398&203426&179955&$\dotsb$& 42883&$\dotsb$& 4350617\\\cline{2-13}
&Total&2320873&&51120&102088&223748&251252&234698&$\dotsb$& 60803&$\dotsb$&10000000\\\cline{2-13}
\end{tabular}

}
\end{table}
%
We used $7.52$ gigahertz-years to compute $10,000,000$ instances of this Schubert problem, all involving secant flags. 
The rows are labeled with the even integers from $0$ to $12$, as the number of real solutions
has the same parity as the number of complex solutions.
The first column with overlap number $0$ represents tests of the Secant Conjecture.
Since the only entry is in the row for $12$ real solutions, the Secant Conjecture was
verified in $2,320,873$ instances. 
The column labeled overlap number $1$ is empty because flags for this problem cannot have
overlap number 1.   
Perhaps the most interesting feature is that for overlap number $2$, all computed solutions
were real, while for overlap number $3$ at least six solutions were real, and for overlap
number 4, at least four were real.
It is only with overlap number 9 and above that we computed an instance with no real solutions.

We also computed $200,000,000$ instances of this same Schubert problem with four secant
flags (for the Schubert variety $X_{\sI}$) and two osculating flags (for the Schubert
variety $X_{\sThI})$.
These data are compiled in Table~\ref{T:Sample_table}.
\begin{table}[htb]
\caption{Experimental data for $\I^6\cdot \sThI^2=12$ with two osculating flags.}
\label{T:Sample_table}
\noindent{\small 

\noindent\begin{tabular}{r|r||r|r|r|r|r|r|r|c||r|}
\multicolumn{11}{c}{Overlap Number}\\
\cline{2-11}
\multirow{10}{*}{\begin{sideways}Real Solutions\end{sideways}} &
\textbackslash & 0&1&2&3&4&5&6&\ $\dotsb$\
&Total\\\cline{2-11}\noalign{\smallskip}\cline{2-11}
& 0&        & &       &       &       &       &       &$\dotsb$&   13894\\\cline{2-11}
& 2&        & &       &       &       &   3799&     19&$\dotsb$& 1357929\\\cline{2-11}
& 4&        & &       &       &  24756&  93214& 186521&$\dotsb$&12146335\\\cline{2-11}
& 6&        & &       &       & 210843& 495977& 731938&$\dotsb$&29925437\\\cline{2-11}
& 8&        & &       &       & 254875& 640663& 508884&$\dotsb$&36708450\\\cline{2-11}
&10&        & &       &       & 153520& 442928& 229530&$\dotsb$&26500908\\\cline{2-11}
&12&49743228& &1171814&2324847&5900258&5944524&3971316&$\dotsb$&93347047\\\cline{2-11}
&Total&49743228&&1171814&2324847&6544252&7621105&5628208&$\dotsb$&200000000\\\cline{2-11}
\end{tabular}
%
%
}
\end{table}
This computation took $261$ gigahertz-days---twenty times as many instances as
Table~\ref{T:X^4X31^2=12} in about one-tenth of the time.
This speed-up occurs because using two osculating flags gives a formulation 
with only four variables instead of 12.
This computation tested the Generalized Secant Conjecture; its
computed instances form the first column.
As the only entry in that column is in the row for $12$ real solutions, the Generalized
Secant Conjecture was verified in $49,743,228$ instances. 
As with Table~\ref{T:X^4X31^2=12} there is visibly an inner border to these data,
but for this computation there are instances with no real solutions starting with overlap
number eight.

%
\subsection{Computing Schubert problems} 
 
A $k\times(n{-}k)$ matrix $X\in\C^{k\times(n{-}k)}$ determines a general point in 
$G(k,n)$, namely the row space $H$ of the $k\times n$ matrix (also written $H$)  
 \begin{equation}\label{Eq:LocCoords} 
   H\ :=\ \bigl(I_k\ \colon\ X\bigr)\,. 
 \end{equation} 
If we represent an $i$-plane $F_i$ as the row space of an  
$i\times n$ matrix $F_i$ of full rank, then  
 \begin{equation}\label{Eq:rankConditions}
   \dim H\cap F_i\ \geq\ j\ \Longleftrightarrow\  
   \rank \left(\begin{matrix}H\\F_i\end{matrix}\right)\  
   \leq\ k+i-j\,, 
 \end{equation}
which is given by the vanishing of all $(k{+}i{-}j{+}1)\times(k{+}i{-}j{+}1)$ subdeterminants. 
We represent a flag $\Fdot$ by a full rank $n\times n$ matrix whose first $i$ rows span $F_i$.
Then~\eqref{Eq:rankConditions} leads to equations for the Schubert 
variety $X_\lambda\Fdot$ in the coordinate patch~\eqref{Eq:LocCoords}. 
In practice, we only need an $(n{-}k{+}i{-}\lambda_i)\times n$ matrix, where $\lambda_i$ is 
the last nonzero part of $\lambda$.
 
To represent a secant $i$-plane, we use an $i\times n$ matrix $F_i(t_1,\dotsc,t_i)$ 
whose $j$th row is the vector  $\gamma(t_j)$, where $t_1,\dotsc,t_i\in \R$, 
and $\gamma(t) = (1,t,\dotsc,t^{n-1})$ is the rational normal curve.  
Similarly, the $i$-plane $F_i(t)$ osculating $\gamma$ at the  
point $\gamma(t)$ is represented by the $i\times n$ matrix whose $j$th row is 
$\gamma^{(j-1)}(t)$.
 
For example, Conjecture~\ref{C:first} involves the Schubert problem $\I^6=5$ on $G(2,5)$ 
where $\I$ is the Schubert condition of a $2$-plane meeting a $3$-plane. 
The solutions are $2$-planes spanned by the first two rows of the matrix in~\eqref{Eq:G25}. 
The last three rows in the matrix are the points $\gamma(s_i)$, $\gamma(t_i)$, $\gamma(u_i)$
that span the $3$-plane of a secant flag. 
 
We use the computer algebra system Singular~\cite{SINGULAR} to compute an 
\demph{eliminant} of the polynomial system modeling a given instance of the Schubert 
problem $\lambda^1,\dotsc,\lambda^m$. 
This is a univariate polynomial $f(x)$ whose roots are all the 
$x$-coordinates of solutions to the Schubert problem in the patch~\eqref{Eq:LocCoords}.  
(See, for example, \cite[Chap.~2]{CLO}.)
By the Shape Lemma~\cite{BMMT}, when the eliminant $f(x)$ has degree equal to
$d(\lambda^1,\dotsc,\lambda^m)$ and is  
square-free, then the solutions to the Schubert problem are in one-to-one correspondence
with the roots of the eliminant $f(x)$, with real roots corresponding to real solutions.
We use the {\tt realroot} command of the mathematical software Maple to compute the number of real
roots of the eliminant $f(x)$. 
 
If the eliminant does not satisfy these hypotheses, then 
we compute an eliminant with respect to a different coordinate
of the patch~\eqref{Eq:LocCoords}.
It is sometimes the case that no coordinate provides a satisfactory eliminant.
This will occur if there is a solution with multiplicity (the Schubert varieties do not
meet transversally) or if the coordinate patch does
not contain all solutions.
In general it will occur when the computed instance lies in a \demph{discriminant
  hypersurface} in the space of all instances.
When developing and testing our software for this experiment, we observed that this
situation was extremely rare, and it only occurred when the overlap number was positive and
there were multiple solutions,
which agrees with the transversality assertion in the Secant Conjecture.
When our software detects that no coordinate provides a satisfactory eliminant, it
deterministically perturbs the points of secancy, preserving the overlap number, and
repeats this elimination procedure. 
This has always worked to give an eliminant satisfying the hypotheses.

As with Tables~\ref{T:X^4X31^2=12} and ~\ref{T:Sample_table}, working in a different set of
local coordinates enables us to efficiently compute  
instances of the Generalized Secant Conjecture~\ref{Co:GSC} for one (and sometimes two)
osculating flags. 
With one flag osculating at $\gamma(\infty)$, we may use local coordinates
as described in~\cite{RSSS}.

With two osculating flags, there is a smaller choice of local coordinates available.
Suppose that $\be_1,\dotsc,\be_n$ are the standard basis vectors corresponding to columns
of our matrices.
Then the flag $\Fdot(\infty)$ osculating the rational normal curve
$\gamma$ at $\gamma(\infty)=\be_n$ and the flag $\Fdot(0)$ osculating at $\gamma(0)=\be_1$ have
\[ 
   F_i(\infty)\ =\ \Span\{\be_{n{+}1{-}i}, \dotsc, \be_{n-1}, \be_n\}
   \qquad\mbox{and}\qquad
   F_i(0)\ =\ \Span\{ \be_1, \be_2,\dotsc, \be_i\}\,. 
\] 
General points in $X_\lambda\Fdot(\infty)\cap X_\mu\Fdot(0)$ are represented by 
$k\times n$ matrices where row $i$ has a 1 in column $\lambda_{k{+}1-i}+i$,
arbitrary entries in subsequent columns up to column $n{-}k{-}1{+}i{-}\mu_i$,
and $0$'s elsewhere.
Here is such a matrix with $k=3$, $n=8$, $\lambda=\TI$, and $\mu=\sThI$:
\[
  \left(\begin{matrix}
    \Blue{1}&\M{*}&\M{*}&\Red{0}&\Red{0}&0&\Red{0}&0\\
    0&\Red{0}&\Blue{1}&\M{*}&\M{*}&\M{*}&\Red{0}&0\\
    0&\Red{0}&0&\Red{0}&\Blue{1}&\M{*}&\M{*}&\M{*}
   \end{matrix}\right) .
\]

%
\subsection{Numerical experimentation}
In~\cite{HS}, 
25,000 instances of the Shapiro Conjecture for the
Schubert problem $\Tw^8=126$ were computed, and for each instance 
the software alphaCertified used 256-bit precision to softly certify that all solutions
were real. 
(A \demph{soft certificate} is one computed with floating point arithmetic that would be rigorous
 if computed with exact rational arithmetic.)
The solutions were computed using the software package Bertini~\cite{BHSW06},
which is based on numerical homotopy continuation~\cite{SW05}.
Given a system of $n$ polynomial equations in $n$ unknowns, Smale's
$\alpha$-theory~\cite{S86} gives algorithms for certifying that Newton iterations 
applied to an approximate solution will converge to a solution,
and also may be used to certify that the solution is real.
As explained in~\cite{HS}, this Schubert problem has such a formulation.
These algorithms are implemented in the software alphaCertified~\cite{HS}.

%
\section{Lower bounds and inner borders}\label{Sec:inner} 
 
The most ubiquitous and enigmatic phenomenon that we have observed in our data is 
the apparent ``inner border'' in many of the tables. 
Typically, we do not observe instances with zero or few real 
solutions when the overlap number is small. 
This is manifested by a prominent staircase separating observed pairs of 
(real solutions, overlap number) from unobserved pairs. 
This feature is clearly visible in Tables~\ref{T:X^4X31^2=12} and~\ref{T:Sample_table},
and in Table~\ref{T:W1^8} for the problem $\I^8=14$ in
$G(2,6)$.  
\begin{table}[htb] 
\caption{Real solutions vs.~overlap number for $\I^8=14$.} 
\label{T:W1^8} 
\noindent{\small  
\noindent\begin{tabular}{|r||r|r|r|r|r|r|r|c||r|}\hline 
\textbackslash & 0&1&2&3&4&5&6&\ $\dotsb$\ &Total\\\hline\hline 
0     &       &&     &      &      &      &      &\ $\dotsb$\ &    4272\\\hline 
2     &       &&     &      &      &      &      &\ $\dotsb$\ &  127217\\\hline 
4     &       &&     &      &   693&  1481&  6660&\ $\dotsb$\ &  879658\\\hline 
6     &       &&     &      &   224&   510&  2541&\ $\dotsb$\ & 2304233\\\hline 
8     &       &&     &      &   526&   939&  3561&\ $\dotsb$\ & 2914837\\\hline 
10    &       &&     &      &  1052&  2074&  6985&\ $\dotsb$\ & 2205198\\\hline 
12    &       &&     &      &  1556&  2595&  7300&\ $\dotsb$\ & 1224667\\\hline 
14    &3328772&&60860&120625&310819&246910&237704&\ $\dotsb$\ & 5339918\\\hline\hline 
Total &3328772&&60860&120625&305870&254509&264751&\ $\dotsb$\ &15000000\\\hline 
\end{tabular}  
}
\end{table} 
There, it is only with overlap number 8 or larger that we observe instances with two
real solutions; and with overlap number 16 or larger, instances with no real solutions. 
(These columns are not displayed for reasons of space.) 
 
This problem involves $2$-planes meeting eight secant $4$-planes. 
There are over $10^{18}$ configurations of eight secant $4$-planes,
and hence it is impossible to systematically study all configurations as in 
Section~\ref{S:chords}. 
This is the case for most of the problems we studied. 
Because of the coarseness of our measure of overlap, we doubt it is possible to formulate a 
meaningful conjecture about this inner border based on our data. 
Nevertheless, we believe that this problem, like the problem of four lines, contains rich 
geometry, with certain configurations having a lower bound on the number of real solutions. 
 
There are many meaningful polynomial systems or geometric problems having a non-zero
lower bound on their number of  real solutions.
These include rational curves interpolating points on toric del {P}ezzo 
surfaces~\cite{IKS03,IKS04,IKS09,Mi05,W}, sparse polynomial 
systems from posets~\cite{JW07,SS}, and some lower bounds in the Schubert
calculus \cite{AzarGab,EG01}.  
 
Lower bounds and inner borders were also observed studying the 
Monotone Conjecture~\cite[\S~3.2.2]{RSSS}. 
The original example of a lower bound was due to Eremenko and Gabrielov~\cite{EG01}.
The Wronskian of linearly independent polynomials $f_1(t),f_2(t),\dotsc,f_k(t)$ of degree
$n{-}1$,   
\[ 
   \W(f_1,f_2,\dotsc,f_k)\ :=\ \det\left(  \begin{matrix} 
           f_1(t) & f'_1(t) & f''_1(t) & \dotsb & f^{(k-1)}_1(t)\\ 
           f_2(t) & f'_2(t) & f''_2(t) & \dotsb & f^{(k-1)}_2(t)\\ 
           \vdots & \vdots  & \vdots  & \ddots & \vdots \\ 
           f_k(t) & f'_k(t) & f''_k(t) & \dotsb & f^{(k-1)}_k(t)\\ 
    \end{matrix}\right) ,
\] 
has degree $k(n{-}k)$, which gives a finite map 
$\W\colon G(k, \C_{n-1}[t])\longrightarrow\P^{k(n-k)}$ 
with the general fiber consisting of $d(\I^{k(n-k)})$ (see~\eqref{Eq:Schubert_number})
linear spaces of polynomials.  
Theorem~\ref{Th:MTV} implies that if
$w(t)$ is a polynomial with $k(n{-}k)$ distinct real roots then each of the $d(\I^{k(n-k)})$
points in the fiber of $\W$ over $w(t)$ is real. 
Eremenko and Gabrielov showed that if $n$ is odd, there is a non-trivial lower bound 
on the number of real spaces of polynomials in the fiber of $\W$ over {\it any} polynomial 
$w(t)$ with real coefficients. 
 
Azar and Gabrielov~\cite{AzarGab} studied the problem $\I^{2n{-}4}$ in 
$G(n{-}2,n)$ of $(n{-}2)$-planes in $\C^n$ which meet one secant line and $2n{-}5$ tangent 
lines.
When the interval of secancy contains no tangent points, this is an instance of the
Generalized Secant Conjecture~\ref{Co:GSC}.
They establish lower bounds on the number of real solutions which depend upon the
configuration of the points of secancy and tangency.

%
\section{Gaps}\label{Sec:gaps}

The Schubert problem $\sTT^4 = 6$ on $G(4,8)$ involves $4$-planes whose intersection with 
each of four general $4$-planes is at least two-dimensional.
We computed $1,000,000$ instances of this problem, obtaining the results in
Table~\ref{T:real gap}. 
\begin{table}[htb]
\caption{Real solutions vs.\ overlap number for $W_{\sTT}^4 = 6$.}
\label{T:real gap}
\noindent \small {
\begin{tabular}{|r || r|r|r|r|r|r|r|r || r|}
\hline
  & 0 & 1 & 2 & 3 & 4 & 5 & 6 & \ $\dotsb$ \ & Total \\ \hline\hline
0 & &  &  &  & & & & & 0 \\  \hline
2 &  &  &  & 1441 & 7730 & 14277 & 16636 & \ $\dotsb$ \ & 147326 \\ \hline
4 & &  &  &  & & & & & 0 \\  \hline
6 & 280304 &  & 13131 & 25708 & 62833 & 55919 & 57719 & \ $\dotsb$ \ & 852674 \\ \hline
\hline
Total & 280304 & 0 & 13131 & 27149 & 70563 & 70196 & 74355 & \ $\dotsb$ \ & 1000000 \\ \hline
\end{tabular}

}
\end{table}
A system of real polynomial equations with 6 solutions can, {\sl a priori}, have $0$, $2$,
$4$, or $6$ real solutions; yet, strikingly, this Schubert problem only has $2$ or $6$
real solutions, never $0$ or $4$. 

Although our observations involved only secant flags, this phenomenon holds for any
real flags. 
As we describe below,
this follows from ideas in Vakil's discussion of this Schubert problem in \cite[\textsection 3.13]{MR2247966}.
(Vakil's discussion, however, focuses on explaining a different phenomenon, namely, Derksen's observation
that the Galois group of this Schubert problem is \demph{deficient},
i.e., smaller than the symmetric or alternating group.)

We consider the \demph{auxiliary Schubert problem} $\Th^4 = 4$ on $G(2,8)$,
counting $2$-planes which meet four general $4$-planes.
Given $4$-planes $W_1,\dotsc,W_4$, let $P_1,\dotsc,P_4$ be the $2$-planes which meet them.
Then the solutions to the original Schubert problem $W_{\sTT}^4=6$ are precisely the $6$ sums
of the form $P_i + P_j$. 
Such a sum is real if and only if $P_i$ and $P_j$ are each real or if $P_i$ and $P_j$ are a
pair of complex conjugate subspaces.

If the $W_i$ are real, then there can be $0$, $1$, or $2$ complex conjugate pairs among
the $P_i$. 
Then the number of solutions $P_i + P_j$ which are real is, respectively, $6$, $2$, and $2$.
This explains the observations in Table~\ref{T:real gap}.

This is the first in a family of Schubert problems in $G(4,2n)$ for $n\geq 4$ with such
gaps in their numbers of real solutions.
These involve enumerating the $4$-planes which have at least a two-dimensional
intersection with each of four general $n$-planes in
$\C^{2n}$.
For each, there is an auxiliary Schubert problem on $G(2,2n)$ of $2$-planes meeting four
general $n$-planes.
This will have $n$ solutions, and the solutions to the original problem are $4$-planes
spanned by pairs of solutions to the auxiliary problem.
The original problem will have $\binom{n}{2}$ solutions, corresponding to pairs of
solutions to the auxiliary problem.
A solution is real either when both elements of the pair are real or when the pair
consists of complex conjugate solutions.
We remark that the auxiliary problem may have any number $r$
of real solutions, where $0\leq r\leq n$ and $n{-}r$ is even---this may be deduced from the
description of the Schubert problem in terms of elementary geometry given, for example,
in~\cite[\S~8.1]{So97a}. 
These restrictions are identical to restrictions on the number of real quadratic factors
of a general real polynomial of degree $n$, as in~\cite[Theorem~7.8]{SS}.
We summarize this discussion.

\begin{theorem}
 The Schubert problem of\/ $4$-planes that have at least a two-dimensional intersection with
 each of four general real $n$-planes in $\C^{2n}$ has $\binom{n}{2}$ solutions.
 The number of real solutions is 
\[
   \binom{r}{2}\ +\ c\,,
\]
 where the auxiliary problem of\/ $2$-planes meeting each of four general real $n$-planes in
 $\C^{2n}$ has $r$ real solutions and $c$ pairs of complex conjugate solutions and 
 $r+2c=n$.
\end{theorem}

\providecommand{\bysame}{\leavevmode\hbox to3em{\hrulefill}\thinspace}
\providecommand{\MR}{\relax\ifhmode\unskip\space\fi MR }
\providecommand{\MRhref}[2]{%
  \href{http://www.ams.org/mathscinet-getitem?mr=#1}{#2}
}
\providecommand{\href}[2]{#2}

\end{document}